\documentclass[a4paper,11pt,reqno]{amsart}
\usepackage[latin1]{inputenc}
\usepackage[T1]{fontenc}
\usepackage[english]{babel}
\usepackage{amsfonts, amsmath, amssymb, mathrsfs}
\usepackage{amsthm}
\usepackage{dsfont}
\usepackage{braket}
\usepackage{color}
\definecolor{darkbrown}{rgb}{.5,.2,0}

\usepackage{graphicx}
\graphicspath{{graphics/}}
\usepackage[colorlinks=true,urlcolor=blue,citecolor=black,linkcolor=black]
{hyperref}
\usepackage{soul} % to cross out text
\usepackage{enumitem}
\setlist[enumerate]{%
topsep=.5ex plus0.5ex minus0.2ex,  %additional space above the list
parsep=.5ex plus0.3ex minus0.1ex,  %space between paragraphs of an item: parsep
itemsep=.5ex plus0.5ex minus0.2ex, %space between items: itemsep+parsep
leftmargin=3ex,    %space to the left of the text in second, third, etc. lines
itemindent=0ex,    %indentation of first line against second line
labelwidth=*,      %width of label box, * is standard width depending on font
labelsep=1ex,      %space between label box and text in first line
listparindent=0ex, %indentation of first line of paragraph within an item
rightmargin=0ex,   %space to the right of the text
align=left,        %label alignment, default is right
font=\rmfamily,
label=\arabic*)}
\theoremstyle{plain}
\newtheorem{Theorem}{Thm}[section]
\newtheorem{Thm}[Theorem]{Theorem}
\newtheorem{Pro}[Theorem]{Proposition}
\newtheorem{Lem}[Theorem]{Lemma}
\newtheorem{Cor}[Theorem]{Corollary}
\theoremstyle{definition}
\newtheorem{Exa}[Theorem]{Example}
\newtheorem{Rem}[Theorem]{Remark}
\newtheorem{Ope}[Theorem]{Open Problem}
\newcommand{\C}{\mathbb{C}}
\newcommand{\N}{\mathbb{N}}
\newcommand{\R}{\mathbb{R}}
\newcommand{\cA}{\mathcal{A}}
\newcommand{\cB}{\mathcal{B}}
\newcommand{\cP}{\mathcal{P}}

\DeclareMathOperator\aff{aff}
\DeclareMathOperator\cc{cc} % convex core
\DeclareMathOperator\cone{cone} 
\DeclareMathOperator\ri{ri}
\DeclareMathOperator\rai{rai}
\DeclareMathOperator\sgn{sgn}
\DeclareMathOperator\spt{spt}
\DeclareMathOperator\dr{\mathrm{d}}
\begin{document}
\title{A note on faces of convex sets}
\author{Stephan Weis}
\begin{abstract}
The faces of a convex set owe their relevance to an interplay between 
convexity and topology that is systematically studied in the work of 
Rockafellar. Infinite-dimensional convex sets are excluded from this 
theory as their relative interiors may be empty. Shirokov and the 
present author answered this issue by proving that every point in a 
convex set lies in the relative algebraic interior of the face it 
generates. This theorem is proved here in a simpler way, connecting 
ideas scattered throughout the literature. This article summarizes 
and develops methods for faces and their relative algebraic interiors 
and applies them to spaces of probability measures.
\end{abstract}
\date{29 September 2024}
\subjclass[2020]{52A05, 46E27}
\keywords{Convex set,
extreme set,
face,
face generated by a point,
relative interior,
relative algebraic interior,
%locally convex space,
%Zermelo-Fraenkel set theory,
Radon-Nikodym derivative,
convex core}
%
% 52 Convex and discrete geometry
%
% 52Axx	General convexity
% 52A05 Convex sets without dimension restrictions (aspects of convex geometry)
%
% 46Exx Linear function spaces and their duals 
%       [See also 30H05, 32A38, 46F05] {For function algebras, see 46J10}
% 46E27 Spaces of measures [See also 28A33, 46Gxx]
% 
% 28A33 Spaces of measures, convergence of measures [See also 46E27, 60Bxx]
%
\maketitle
%
%%%%%%%%%%%%%%%%%%%%%%%%%%%%%%%%%%%%%%%%%%%%%%%%%%%%%%%%%%%%%%%%%%%%%%%%%%%%
%%%%%%%%%%%%%%%%%%%%%%%%%%%%%%%%%%%%%%%%%%%%%%%%%%%%%%%%%%%%%%%%%%%%%%%%%%%%
%%%%%%%%%%%%%%%%%%%%%%%%%%%%%%%%%%%%%%%%%%%%%%%%%%%%%%%%%%%%%%%%%%%%%%%%%%%%
%%%%%%%%%%%%%%%%%%%%%%%%%%%%%%%%%%%%%%%%%%%%%%%%%%%%%%%%%%%%%%%%%%%%%%%%%%%%
%%%%%%%%%%%%%%%%%%%%%%%%%%%%%%%%%%%%%%%%%%%%%%%%%%%%%%%%%%%%%%%%%%%%%%%%%%%%
%
\hfill
Dedicated to R.\,Tyrrell~Rockafellar on His Ninetieth Birthday
\par
%
%%%%%%%%%%%%%%%%%%%%%%%%%%%%%%%%%%%%%%%%%%%%%%%%%%%%%%%%%%%%%%%%%%%%%%%%%%%%
%%%%%%%%%%%%%%%%%%%%%%%%%%%%%%%%%%%%%%%%%%%%%%%%%%%%%%%%%%%%%%%%%%%%%%%%%%%%
%%%%%%%%%%%%%%%%%%%%%%%%%%%%%%%%%%%%%%%%%%%%%%%%%%%%%%%%%%%%%%%%%%%%%%%%%%%%
%%%%%%%%%%%%%%%%%%%%%%%%%%%%%%%%%%%%%%%%%%%%%%%%%%%%%%%%%%%%%%%%%%%%%%%%%%%%
%%%%%%%%%%%%%%%%%%%%%%%%%%%%%%%%%%%%%%%%%%%%%%%%%%%%%%%%%%%%%%%%%%%%%%%%%%%%
%
\section{Introduction}
A \emph{face} of a convex set $K$ in a real vector space $V$ is a convex 
subset of $K$ including every pair of points in $K$ that are the endpoints 
of some open segment intersected by this convex subset. One-point faces 
are in a one-to-one correspondence with extreme points and are useful in 
functional analysis 
\cite{Alfsen1971,Barvinok2002,Bourbaki1987,Holmes1975,Werner2018}. Larger 
faces play a minor role beyond finite dimensions. A notable exception 
is Alfsen and Shultz' work on operator algebras \cite{AlfsenShultz2001}.
\par
Faces of finite-dimensional convex sets owe their success to an interplay 
of convexity and topology that appears in Grünbaum's work
\cite[Sec.~2.4]{Gruenbaum2003} and is systematically studied by Rockafellar 
\cite{Rockafellar1970}. A central notion is the \emph{relative interior} 
$\ri(K)$ of $K$, the interior of $K$ in the induced topology on the affine 
hull $\aff(K)$ of $K$. Lacking monotonicity, the operator $\ri$ is not the 
interior operator of a topology. Imagine a side $K$ of a triangle $L$ in 
the Euclidean plane $V$; although $K\subset L$ holds, $\ri(K)$ and $\ri(L)$ 
are disjoint and nonempty. Rockafellar and Wets 
\cite[p.~75]{RockafellarWets2009} stress that the closure of $\ri(K)$ 
includes $K$ in a Euclidean space $V$\!. Notably, $K\neq\emptyset$ implies 
$\ri(K)\neq\emptyset$.
\par
The last assertion is false in a Hausdorff topological vector space, as the 
interior of $K$ can be empty and the affine hull of $K$ can still be equal to $V$\!. 
\begin{enumerate}
\item
A first example is the convex set $K=\{x\in V:\ell(x)> 0\}$ defined by a 
discontinuous linear functional $\ell:V\to\R$, because $\ker(\ell)$ 
is dense in $V$\!. 
\item
The closed convex set $K=\{f\in V : \mbox{$f\geq 0$ a.e.} \}$ in the
Banach space $V=L^p([0,1])$ of $p$-integrable real functions, 
$1\leq p<\infty$, has no interior points. Every function $f\in K$ is the 
limit of $f\cdot 1_{(1/n,1]}-1_{[0,1/n]}$ as $n\to\infty$, where $1_A$ 
denotes the indicator function on a measurable set $A$. 
\end{enumerate}
Avoiding empty interiors, Borwein and Goebel \cite[Thm.~2.8]{BorweinGoebel2003} 
study modified interiors of convex sets in Banach spaces. This article, instead, 
focuses on those faces of a convex set that happen to have nonempty relative 
interiors. 
\par
Among all vector topologies on $V$\!, the greatest interior (with respect to 
inclusion) of a convex set is achieved by the 
\emph{finest locally convex topology} $\mathfrak{T}_\omega(V)$. We deduce this
from the continuity of the map $\R\to V$\!, $\lambda\mapsto x+\lambda v$, for 
every $x,v\in V$\!, bearing in mind that a point $x\in K$ lies in the interior 
of $K$ for $\mathfrak{T}_\omega(V)$, if and only if for every line $g$ in $V$ 
containing $x$, the intersection $g\cap K$ includes an open segment 
containing $x$ \cite[II.26]{Bourbaki1987}. The set of all such points $x$ is 
called the \emph{algebraic interior} \cite{Barvinok2002,Clason2020} or ``core'' 
\cite{Clason2020,Holmes1975} of~$K$. 
\par
Example 1) above ceases to exist in the topology $\mathfrak{T}_\omega(V)$, 
which renders every linear functional continuous \cite[II.26]{Bourbaki1987}. 
The empty interior of Example~2) persists in the topology 
$\mathfrak{T}_\omega(V)$, as the interior of every closed convex subset of 
a Banach space is the algebraic interior of that set 
\cite[p.~46]{Clason2020}. Another example with empty algebraic interior is
the set of univariate polynomials with real coefficients and positive 
leading coefficients, which Barvinok examines in several revealing 
exercises \cite[III.1.6]{Barvinok2002}. By contrast, if $\dim(V)<\infty$ 
then $\mathfrak{T}_\omega(V)$ is the Euclidean topology and, again, 
$K\neq\emptyset$ implies $\ri(K)\neq\emptyset$. 
\par
The relative interior $\ri(K)$ for the topology $\mathfrak{T}_\omega(V)$ 
can be described in terms of the Euclidean topology on a line. 
Shirokov and the present author \cite{WeisShirokov2021} define the 
\emph{relative algebraic interior} $\rai(K)$ of $K$ as the set of points 
$x\in K$ such that for every line $g$ in $\aff(K)$ containing $x$, 
the intersection $g\cap K$ includes an open segment containing $x$. This 
definition differs from that of the algebraic interior just in the affine 
space confining the lines. Holmes \cite{Holmes1975} calls $\rai(K)$ the 
``intrinsic core''. We prove $\ri(K)=\rai(K)$ for $\mathfrak{T}_\omega(V)$ 
in Sec.~\ref{sec:fintest-loc-conv-topo}. Topological vector spaces and 
relative interiors are ignored subsequently; instead, relative algebraic 
interiors are used consistently.
\par
In their studies of constrained density operators \cite{WeisShirokov2021}, 
Shirokov and this author employ the \emph{face of $K$ generated by} 
a point $x\in K$. This is the smallest face of $K$ containing $x$, which we 
denote by $F_K(x)$. A basic property is that $x\in\rai(F_K(x))$ holds 
for all $x\in K$.
\par
This key result is deduced in Sec.~\ref{sec:face-generated-by-a-point} from two 
elementary assertions, whereas our prior proof unnecessarily employs the 
Kuratowski-Zorn lemma. The first elementary assertion (Prop.~\ref{pro:Alfsen}) 
is Alfsen's formula \cite[p.~121]{Alfsen1971} 
\begin{equation}\label{eq:Alfsen}
F_K(x)=
\left\{ y\in K \mid \exists\epsilon>0 \colon x+\epsilon(x-y)\in K \right\}\,\text{.}
\end{equation}
The second one (Prop.~\ref{pro:BorweinGoebel}) is Borwein and Goebel's
observation \cite{BorweinGoebel2003} that a point $x\in K$ lies in 
$\rai(K)$ if and only if for all $y\in K$ there is $\epsilon>0$ such that 
$x+\epsilon(x-y)\in K$. A point $x$ satisfying the latter proposition is 
called an \emph{internal point} \cite{Dubins1962} or a 
``relatively absorbing point'' \cite{BorweinGoebel2003} of $K$.
\par
Consequences of $x\in\rai(F_K(x))$, $x\in K$, are organized roughly as 
follows. A review in Sec.~\ref{sec:prior-work} and new findings in 
Sec.~\ref{sec:novel-results-I} extend selected results from Secs.~6 
and~18 of Rocka\-fellar's monograph \cite{Rockafellar1970}. 
Secs.~\ref{sec:novel-results-II} and~\ref{sec:Dubins-Face} refer to 
\mbox{Secs.~2--4} of Dubins' paper \cite{Dubins1962} on infinite-dimensional 
convexity.
\par
The methods of this paper are suitable to study the space of probability 
measures on a measurable space. The face generated by a probability 
measure is described in Sec.~\ref{sec:example-prob-measures}. The face 
generated by a Borel probability measure $\mu$ on $\R^d$ is related to 
Csisz{\'a}r and Mat{\'u}{\v s}' notion \cite{CsiszarMatus2001} of the 
\emph{convex core} of $\mu$ in Sec.~\ref{sec:example-convex-cores}.
Sec.~\ref{sec:example-countable} studies measures on the set of natural numbers.
\par
%
%
%%%%%%%%%%%%%%%%%%%%%%%%%%%%%%%%%%%%%%%%%%%%%%%%%%%%%%%%%%%%%%%%%%%%%%%%%%%%
%%%%%%%%%%%%%%%%%%%%%%%%%%%%%%%%%%%%%%%%%%%%%%%%%%%%%%%%%%%%%%%%%%%%%%%%%%%%
%%%%%%%%%%%%%%%%%%%%%%%%%%%%%%%%%%%%%%%%%%%%%%%%%%%%%%%%%%%%%%%%%%%%%%%%%%%%
%%%%%%%%%%%%%%%%%%%%%%%%%%%%%%%%%%%%%%%%%%%%%%%%%%%%%%%%%%%%%%%%%%%%%%%%%%%%
%%%%%%%%%%%%%%%%%%%%%%%%%%%%%%%%%%%%%%%%%%%%%%%%%%%%%%%%%%%%%%%%%%%%%%%%%%%%
%
\section{Main definitions}
\label{sec:definitions}
Throughout this paper, $V$ denotes a real vector space and $K$ a convex 
subset of $V$\!, unless stated otherwise. If $x\neq y$ are distinct 
points of $V$\!, then 
\begin{align*}
(x,y) &= \left\{(1-\lambda)x+\lambda y\colon \lambda\in(0,1)\right\}\\
\text{resp.}\quad 
[x,y] &= \left\{(1-\lambda)x+\lambda y\colon \lambda\in[0,1]\right\}
\end{align*}
is called the \emph{open segment} resp.\ \emph{closed segment} with 
endpoints $x,y$. Each of the symbols $(x,x)=[x,x]$ denotes the singleton 
$\{x\}$ containing $x\in V$\!.
\par
An \emph{extreme set} \cite{AliprantisBorder2006} of $K$ is a subset $E$ of 
$K$ including the closed segment $[x,y]$ for all points $x\neq y$ in $K$ 
for which the open segment $(x,y)$ intersects $E$. A point $x\in K$ is an 
\emph{extreme point} of $K$ if $\{x\}$ is an extreme set of $K$. A 
\emph{face} of $K$ is a convex extreme set of $K$. Clearly, any union 
or intersection of extreme sets of $K$ is an extreme set of $K$. 
Since any intersection of convex sets is convex, any intersection of 
faces of $K$ is a face of $K$. In particular, the intersection of all 
faces containing a point $x\in K$ is a face of $K$, which is called the 
\emph{face of $K$ generated by $x$}, and which is denoted by $F_K(x)$.
\par
The \emph{algebraic interior} of $K$ is the set of all points $x$ in $K$ 
such that for every line $g$ in $V$ containing $x$, the intersection 
$g\cap K$ includes an open segment containing $x$ \cite{Barvinok2002}. 
The convex set $K$ is \emph{algebraically open} if it is equal to its
algebraic interior. The \emph{relative algebraic interior} $\rai(K)$ of 
$K$ is the set of all points $x$ in $K$ such that for every line $g$ in 
$\aff(K)$ containing $x$, the intersection $g\cap K$ includes an open 
segment containing $x$ \cite{WeisShirokov2021}. The convex set $K$ is 
\emph{relative algebraically open} if $K=\rai(K)$. A point $x\in K$ is an 
\emph{internal point} \cite{Dubins1962} of $K$ if for all $y\neq x$ in $K$ 
there is $\epsilon>0$ such that $x+\epsilon(x-y)\in K$.
\par
%
%%%%%%%%%%%%%%%%%%%%%%%%%%%%%%%%%%%%%%%%%%%%%%%%%%%%%%%%%%%%%%%%%%%%%%%%%%%%
%%%%%%%%%%%%%%%%%%%%%%%%%%%%%%%%%%%%%%%%%%%%%%%%%%%%%%%%%%%%%%%%%%%%%%%%%%%%
%%%%%%%%%%%%%%%%%%%%%%%%%%%%%%%%%%%%%%%%%%%%%%%%%%%%%%%%%%%%%%%%%%%%%%%%%%%%
%%%%%%%%%%%%%%%%%%%%%%%%%%%%%%%%%%%%%%%%%%%%%%%%%%%%%%%%%%%%%%%%%%%%%%%%%%%%
%%%%%%%%%%%%%%%%%%%%%%%%%%%%%%%%%%%%%%%%%%%%%%%%%%%%%%%%%%%%%%%%%%%%%%%%%%%%
%
\section{The finest locally convex topology}
\label{sec:fintest-loc-conv-topo}
One of the innovations of this paper is that Thm.~\ref{thm:WeisShirokov} is a 
theorem of Zermelo-Fraenkel set theory. However, the use of the relative 
interior in a topological vector space creates a new dependence on the axiom 
of choice by Rem.~\ref{rem:choice}. Coro.~\ref{cor:ri=rai} allows us to avoid
this problem simply by dismissing topological vector spaces altogether, and 
relying on relative algebraic interiors of convex sets instead. This comes at 
the modest price that some theorems of topological vector spaces require a 
proof in this paper. For example, the theorem that the interior of a convex 
set is convex \cite[II.14]{Bourbaki1987} would make Lemma~\ref{lem:ri-convex} 
superfluous. The assertion that the interior of a set is open would essentially 
supersede Thm.~\ref{thm:ri2=ri} (up to questions regarding affine hulls).
\par
A \emph{fundamental system of neighborhoods} of a point $x$ in a topological 
space is any set $\mathfrak{S}$ of neighborhoods of $x$ such that for 
each neighborhood $U$ of $x$ there is a neighborhood $W\in\mathfrak{S}$ such 
that $W\subset U$. A topological real vector space is \emph{locally convex} 
if there exists a fundamental system of neighborhoods of $0$ that are convex 
sets \cite[II.23]{Bourbaki1987}.
\par
It is known that a convex subset $C\subset V$ is open for the finest locally 
convex topology $\mathfrak{T}_\omega(V)$ on $V$ if and only if $C$ is 
algebraically open; a subset $U\subset V$ is open for 
$\mathfrak{T}_\omega(V)$ if and only if $U$ is a union of convex open subsets 
of $V$ \cite[II.26]{Bourbaki1987}. Every nonempty affine subspace $A\subset V$ 
is isomorphic to the vector space of translations $A-A=\{x-y\mid x,y\in A\}$ by 
means of the affine isomorphism $\alpha:x\mapsto x-x_0$ defined by some point 
$x_0\in A$. A topology is defined on $A$ for which a subset $U\subset A$ is open 
if and only if $\alpha(U)$ is open for $\mathfrak{T}_\omega(A-A)$. This topology 
does not depend on $x_0$ and is denoted by $\mathfrak{T}_\omega(A)$. The topology 
induced on $A$ by $\mathfrak{T}_\omega(V)$ is denoted by 
$\mathfrak{T}_\omega(V)|A$. By definition, a subset $U\subset A$ is open for 
$\mathfrak{T}_\omega(V)|A$ if and only if there exists an open set $W$ for 
$\mathfrak{T}_\omega(V)$ such that $U=A\cap W$. Let $S,S'\subset V$ be linear 
subspaces. Then $S'$ is a \emph{complementary subspace of $S$ in $V$} if
$S\cap S'=\{0\}$ and if for all $x\in V$ there is $s\in S$ and $s'\in S'$ such
that $x=s+s'$.
\par
\begin{Pro}\label{pro:char-ri-omega}
If $A$ is an affine subspace of $V$\!, then 
$\mathfrak{T}_\omega(A)=\mathfrak{T}_\omega(V)|A$.
\end{Pro}
\begin{proof}
Using the isomorphism $A\cong A-A$, we assume that $0$ lies in $A$. 
Clearly, the induced topology $\mathfrak{T}_\omega(V)|A$ is locally convex, 
which shows that $\mathfrak{T}_\omega(A)$ is finer than 
$\mathfrak{T}_\omega(V)|A$. Conversely, let $U$ be a convex open subset of 
$A$ for the topology $\mathfrak{T}_\omega(A)$. Let $B$ be a complementary
subspace of $A$ in $V$\!. It is easy to see that the convex set 
$U+B=\{u+b:u\in U, b\in B\}$ is open for $\mathfrak{T}_\omega(V)$ and that 
$A\cap (U+B)=U$ holds. This shows that $\mathfrak{T}_\omega(V)|A$ is finer 
than $\mathfrak{T}_\omega(A)$ and completes the proof.
\end{proof}
\begin{Rem}\label{rem:choice}
The complementary subspace used in Prop.~\ref{pro:char-ri-omega} exist by the 
axiom of choice. Conversely, the existence of a complementary subspace for all 
subspaces of all vector spaces over the reals (or over any other field) 
implies a weak version of the axiom of choice \cite[Lemma~2]{Bleicher1964} 
that is equivalent to the full axiom of choice in Zermelo-Fraenkel set theory 
\cite[Thm.~9.1]{Jech2008}.
\end{Rem}
\begin{Cor}\label{cor:ri=rai}
We have $\ri(K)=\rai(K)$ for the topology $\mathfrak{T}_\omega(V)$. 
\end{Cor}
\begin{proof}
This follows immediately from Prop.~\ref{pro:char-ri-omega}.
\end{proof}
%
%%%%%%%%%%%%%%%%%%%%%%%%%%%%%%%%%%%%%%%%%%%%%%%%%%%%%%%%%%%%%%%%%%%%%%%%%%%%
%%%%%%%%%%%%%%%%%%%%%%%%%%%%%%%%%%%%%%%%%%%%%%%%%%%%%%%%%%%%%%%%%%%%%%%%%%%%
%%%%%%%%%%%%%%%%%%%%%%%%%%%%%%%%%%%%%%%%%%%%%%%%%%%%%%%%%%%%%%%%%%%%%%%%%%%%
%%%%%%%%%%%%%%%%%%%%%%%%%%%%%%%%%%%%%%%%%%%%%%%%%%%%%%%%%%%%%%%%%%%%%%%%%%%%
%%%%%%%%%%%%%%%%%%%%%%%%%%%%%%%%%%%%%%%%%%%%%%%%%%%%%%%%%%%%%%%%%%%%%%%%%%%%
%
\section{Every point lies in the relative algebraic interior of the 
face it generates}
\label{sec:face-generated-by-a-point}
This section provides proofs for Alfsen's formula \eqref{eq:Alfsen}
describing the face generated by a point, and Borwein and Goebel's 
observation \cite{BorweinGoebel2003} that a point is an internal 
point of $K$ if and only if it lies in $\rai(K)$. An immediate 
corollary is Shirokov and the present author's theorem \cite{WeisShirokov2021} 
that every point lies in the relative algebraic interior of the face it 
generates. Let $x\in K$ and define
\begin{align*}
S_K(x) &= 
\left\{ y\in K \mid \exists\epsilon>0 \colon x+\epsilon(x-y)\in K \right\}
\,\text{,}\\
C_K(x) &= 
\left\{ y\in V \mid \exists\epsilon>0 \colon x-\epsilon(x-y)\in K \right\}
\,\text{,}\\
A_K(x) &= 
\left\{ y\in V \mid \exists\epsilon>0 \colon x\pm\epsilon(x-y)\in K \right\}
\,\text{.}
\end{align*}
\par
\begin{figure}
a)\includegraphics[height=3.9cm]{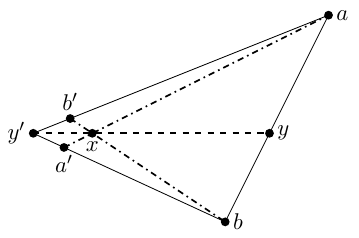}%
b)\includegraphics[height=3.9cm]{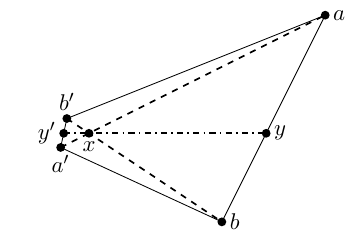}%
\caption{\label{fig:points-in-face-generated}%
Configurations (in the plane) for Prop.~\ref{pro:Alfsen}.
The set $S_K(x)$ is a) an extreme set and b) convex.}
\end{figure}%
\begin{Pro}[Alfsen]\label{pro:Alfsen}
For all $x\in K$ we have $F_K(x)=S_K(x)$.
\end{Pro}
\begin{proof}
Let $y$ be a point in $S_K(x)$ and $y\neq x$. Then there is $\epsilon>0$ such 
that $x$ lies in the open segment with endpoints $y$ and $x+\epsilon(x-y)$. 
As $F_K(x)$ is an extreme set containing $x$, it follows $y\in F_K(x)$, which 
proves $S_K(x)\subset F_K(x)$.
\par
We finish the proof by showing that $S_K(x)$ is a face of $K$. Thus, we consider
distinct points $a\neq b$ in $K$ and a point $y$ in the open segment $(a,b)$, 
see Fig.~\ref{fig:points-in-face-generated}. Let $\eta\in(0,1)$
such that $y=(1-\eta)a+\eta b$. The coefficients for the following constructions 
are obtained from Menelaus' theorem \cite{Barth2004,GruenbaumShephard1995}. 
\par
We show that $S_K(x)$ is an extreme set, see 
Fig.~\ref{fig:points-in-face-generated}~a). Assuming $y\in S_K(x)$, there is 
$\epsilon>0$ such that $y':=x+\epsilon(x-y)$ lies in $K$. Let 
$\epsilon_a=\epsilon(1-\eta)/(1+\epsilon\eta)$ and 
$\xi_a=\epsilon\eta/(1+\epsilon\eta)$. Then 
\[
a':=x+\epsilon_a(x-a)=(1-\xi_a)y'+\xi_a b
\]
lies in $K$, and hence $a\in S_K(x)$. Similarly, $b\in S_K(x)$.
\par
We show that $S_K(x)$ is convex, see 
Fig.~\ref{fig:points-in-face-generated}~b). 
We assume $a,b\in S_K(x)$. Let $\epsilon_a,\epsilon_b>0$ such that
$a':=x+\epsilon_a(x-a)$ and $b':=x+\epsilon_b(x-b)$ lie in $K$;
and let 
\[\textstyle
\epsilon=\frac{\epsilon_a\epsilon_b}{(1-\eta)\epsilon_b+\eta\epsilon_a}
\quad\text{and}\quad
\xi=\frac{\eta\epsilon_a}{(1-\eta)\epsilon_b+\eta\epsilon_a}
\,\text{.}
\]
Then
\[
y':=x+\epsilon(x-y)=(1-\xi)a'+\xi b'
\]
lies in $K$, and hence $y\in S_K(x)$.
\end{proof}
The union in Coro.~\ref{cor:facex} extends over all closed segments in $K$,
whose respective open segments contain $x$, and the singleton 
$\{x\}=(x,x)=[x,x]$.
\par
\begin{Cor}\label{cor:facex}
Every point $x\in K$ is an internal point of $F_K(x)$ and we have
$F_K(x)=\bigcup_{y,z\in K,x\in(y,z)}[y,z]$.
\end{Cor}
\begin{proof}
The first assertion follows from the definition of $S_K(x)$ and from
Prop.~\ref{pro:Alfsen}, which proves $F_K(x)=S_K(x)$. Regarding the 
second assertion, the inclusion ``$\supset$'' holds because $F_K(x)$ 
is an extreme set of $K$ containing $x$. The inclusion ``$\subset$'' 
is implied by a proof of $S_K(x)\subset\bigcup_{y,z\in K,x\in(y,z)}[y,z]$. 
Let $y\in S_K(x)$. Then there is $\epsilon>0$ such that 
$z:=x+\epsilon(x-y)\in K$. Hence, 
$x=\frac{1}{\epsilon+1}(\epsilon y+z)\in(y,z)$ and $y\in[y,z]$ complete
the proof.
\end{proof}
Note that $C_K(x)=\cone(K-x)+x$ holds for all $x\in K$, where $\cone(C)$ 
denotes the set $\bigcup_{\lambda\geq 0}\{\lambda x: x\in C\}$ for every 
convex set $C\subset V$ containing the origin. Borwein and Goebel 
\cite[p.~2544]{BorweinGoebel2003} suggest that a preliminary step to 
Prop.~\ref{pro:BorweinGoebel} should be a proof that $x$ is an internal 
point of $K$ if and only if $\aff(K)=C_K(x)$ holds. Instead, we use the 
following Lemma~\ref{lem:BorweinGoebel}.
\par
\begin{Lem}\label{lem:BorweinGoebel}
Let $x$ be an internal point of $K$. Then $\aff(K)\subset A_K(x)$ holds.
\end{Lem}
\begin{proof}
Let $y\neq x$ be a point in $\aff(K)$. It suffices to find $\epsilon>0$ such 
that the two points $x\pm\epsilon(x-y)$ are both contained in $K$.
\par
Since $x,y\in\aff(K)$ and since $y\neq x$, there exist $y_i\in K$ and 
$\alpha_i\in\R$, $i=1,\dots,n$, not all numbers $\alpha_i$ being zero, 
such that
\[\textstyle
y=x+\sum_i\alpha_iy_i
\quad\text{and}\quad
\sum_i\alpha_i=0\,\text{.}
\]
By the assumption that $x$ is an internal point of $K$, there is $\epsilon_i>0$ 
such that 
\[
y_i':=x+\epsilon_i(x-y_i)
\]
is contained in $K$, $i=1,\dots,n$. Let $\|\alpha\|_1=\sum_{i=1}^n|\alpha_i|$ 
and let $\epsilon$ be the minimum of $\min_i\epsilon_i/\|\alpha\|_1$ and 
$1/\|\alpha\|_1$. Then $\epsilon$ is strictly positive. We have 
\begin{align*}\textstyle
x\pm\epsilon(x-y) 
 & \textstyle
 =x\mp\epsilon\sum_i\alpha_iy_i
 =x\pm\epsilon\sum_i\alpha_i(x-y_i)\\
 &\textstyle
 =\sum_i\frac{|\alpha_i|}{\|\alpha\|_1}
 \underbrace{\left(x\pm\sgn(\alpha_i)\epsilon\|\alpha\|_1(x-y_i)\right)}_{z_i:=}
\,\text{.}
\end{align*}
If $\pm\sgn(\alpha_i)=-1$, then $z_i\in[y_i,x]\subset K$ holds because of
$\epsilon\|\alpha\|_1\leq 1$. If $\sgn(\alpha_i)=0$, then $z_i=x\in K$. If 
$\pm\sgn(\alpha_i)=+1$, then $z_i\in[x,y_i']\subset K$ because 
$\epsilon\|\alpha\|_1\leq\epsilon_i$. This shows that the points 
$x\pm\epsilon(x-y)$ are convex combinations of points in $K$, and therefore 
are themselves points in $K$.
\end{proof}
Prop.~\ref{pro:BorweinGoebel} is mentioned on p.~2544 of \cite{BorweinGoebel2003}.
\par
\begin{Pro}[Borwein-Goebel]\label{pro:BorweinGoebel}
For all $x\in K$, the following assertions are equivalent.
\begin{enumerate}
\item
The point $x$ is an internal point of $K$.
\item
We have $x\in\rai(K)$.
\item
We have $A_K(x)=C_K(x)$.
\item
We have $C_K(x)=\aff(K)$.
\end{enumerate}
\end{Pro}
\begin{proof}
If $x$ is an internal point of $K$, then Lemma~\ref{lem:BorweinGoebel} shows 
$\aff(K)\subset A_K(x)$. This implies 2) by the definition of the relative 
algebraic interior. It also implies 3) and 4), because the inclusions 
$A_K(x)\subset C_K(x)\subset \aff(K)$ are trivial. As \mbox{2) $\Rightarrow$ 1)} 
is clear, it suffices to prove \mbox{n) $\Rightarrow$ 1)} for $n=3,4$.
\par
Assume 3) is true and let $y\neq x$ be a point of $K$. Since $K\subset C_K(x)$,
the point $y$ lies in $A_K(x)$. This provides $\epsilon>0$ such that 
$x+\epsilon(x-y)$ lies in $K$. Hence, $x$ is an internal point of $K$. 
\par
Assume 4) is true and let $y\neq x$ be a point of $K$. Then $z:=2x-y$ lies in 
$\aff(K)$ and hence in $C_K(x)$. This provides $\epsilon>0$ such that
\[
x+\epsilon(x-y)=x-\epsilon(x-z)\in K
\]
and shows that $x$ is an internal point of $K$.
\end{proof}
This section's main result is a novel proof for Thm.~2.3 in 
\cite{WeisShirokov2021}, which states the following.
\par
\begin{Thm}[W\!.-Shirokov]\label{thm:WeisShirokov}
For all $x\in K$ we have $x\in\rai(F_K(x))$.
\end{Thm}
\begin{proof}
The point $x$ is an internal point of $F_K(x)$ by Coro.~\ref{cor:facex}.
It lies in the relative algebraic interior of $F_K(x)$ by 
Prop.~\ref{pro:BorweinGoebel}.
\end{proof}
\begin{Cor}\label{cor:ShirokovWeis-affine}
For all $x\in K$ we have $\aff\left(F_K(x)\right)=A_K(x)$.
\end{Cor}
\begin{proof}
Thm.~\ref{thm:WeisShirokov} and Prop.~\ref{pro:BorweinGoebel} prove
$\aff\left(F_K(x)\right)=A_{F_K(x)}(x)$. The inclusion 
$A_{F_K(x)}(x)\subset A_K(x)$ is clear. Conversely, if $y\in A_K(x)$, then 
there is $\epsilon>0$ such that $x\pm\epsilon(x-y)\in K$. Since $F_K(x)$ 
is an extreme set of $K$ containing $x$, this implies 
$x\pm\epsilon(x-y)\in F_K(x)$, and hence $y\in A_{F_K(x)}(x)$.
\end{proof}
%
%%%%%%%%%%%%%%%%%%%%%%%%%%%%%%%%%%%%%%%%%%%%%%%%%%%%%%%%%%%%%%%%%%%%%%%%%%%%
%%%%%%%%%%%%%%%%%%%%%%%%%%%%%%%%%%%%%%%%%%%%%%%%%%%%%%%%%%%%%%%%%%%%%%%%%%%%
%%%%%%%%%%%%%%%%%%%%%%%%%%%%%%%%%%%%%%%%%%%%%%%%%%%%%%%%%%%%%%%%%%%%%%%%%%%%
%%%%%%%%%%%%%%%%%%%%%%%%%%%%%%%%%%%%%%%%%%%%%%%%%%%%%%%%%%%%%%%%%%%%%%%%%%%%
%%%%%%%%%%%%%%%%%%%%%%%%%%%%%%%%%%%%%%%%%%%%%%%%%%%%%%%%%%%%%%%%%%%%%%%%%%%%
%
\section{Review on the face generated by a point}
\label{sec:prior-work}
Here we review some of our prior work from \cite[Sec.~2]{WeisShirokov2021}.
Lemma~\ref{lem:relint}.1 matches \cite[Thm.~18.1]{Rockafellar1970}.
See Lemma 2.1 in \cite{WeisShirokov2021} for a proof.
\par
\begin{Lem}\label{lem:relint}
Let $C\subset K$ be a convex subset of $K$, 
let $E\subset K$ be an extreme set of $K$,
let $F\subset K$ be a face of $K$, 
and let $x\in K$ be a point in $K$. Then
\begin{enumerate}
\item
$\rai(C)\cap E\neq\emptyset\implies C\subset E$,
\item
$x\in F\iff F_K(x)\subset F$,
\item
$x\in\rai(F)\implies F=F_K(x)$.
\end{enumerate}
\end{Lem}
Lemma~\ref{lem:ri-convex} is proved in Lemma 2.2 in \cite{WeisShirokov2021}.
\par
\begin{Lem}\label{lem:ri-convex}
The complement $K\setminus\rai(K)$ of the relative algebraic interior $\rai(K)$ 
is an extreme set of $K$ and $\rai(K)$ is a convex set. 
\end{Lem}
Whereas Lemma~\ref{lem:relint} and Lemma~\ref{lem:ri-convex} are rather easy to 
prove, the remainder of this section relies on Thm.~\ref{thm:WeisShirokov}.
\par
\begin{Cor}\label{cor:extreme}
Let $S\subset K$. The following assertions are equivalent. 
\begin{enumerate}
\item
$S$ is an extreme set of $K$.
\item
$S$ includes the face $F_K(x)$ of $K$ generated by any point $x$ in $S$.
\item
$S$ is the union of the faces $F_K(x)$ of $K$ generated by the points $x$ in $S$.
\end{enumerate}
\end{Cor}
\begin{proof}
See Coro.~2.5 in \cite{WeisShirokov2021}; a proof is provided for easy reference.
1)~$\Rightarrow$~2) follows from Lem\-ma~\ref{lem:relint}.1 as $x\in\rai(F_K(x))$ 
holds for all $x\in S$ by Thm.~\ref{thm:WeisShirokov}. 
2)~$\Rightarrow$~3) follows from $x\in F_K(x)$ for all $x\in S$.
3)~$\Rightarrow$~1) follows directly from the definition of an extreme set.
\end{proof}
Coro.~\ref{cor:partition}, and Thm.~\ref{thm:maximal-elements} below, 
match \cite[Thm.~18.2]{Rockafellar1970}. Let
\begin{align*}
\mathfrak{U}_1 &= \left\{\rai(F_K(x)) \colon x \in K\right\}\,\text{,}\\
\mathfrak{U}_2 &= 
\left\{\rai(F)\colon\text{$F$ is a face of $K$}\right\}\setminus\{\emptyset\}
\,\text{.}
\end{align*}
We recall that a partition of $K$ is a family of nonempty subsets of $K$
whose elements are mutually disjoint and whose union is $K$.
\par
\begin{Cor}\label{cor:partition}
We have $\mathfrak{U}_1=\mathfrak{U}_2$, the family $\mathfrak{U}_2$ is a 
partition of $K$, and
\[
\left\{\text{$F$ is a face of $K$} \colon \rai(F)\neq\emptyset\right\}
\longrightarrow \mathfrak{U}_2,
\quad 
F\longmapsto\rai(F)
\]
is a bijection.
\end{Cor}
\begin{proof}
See Coro.~2.6 in \cite{WeisShirokov2021}; a proof is provided for easy reference. 
The union of the family $\mathfrak{U}_1$ covers $K$ as $x\in\rai(F_K(x))$ by 
Thm.~\ref{thm:WeisShirokov}. Since $\mathfrak{U}_1\subset \mathfrak{U}_2$ is clear, 
proving that the elements of $\mathfrak{U}_2$ are mutually disjoint implies that
$\mathfrak{U}_1=\mathfrak{U}_2$ and that $\mathfrak{U}_2$ is a partition of $K$. 
Let $F,G$ be faces of $K$ and let $x\in\rai(F)\cap\rai(G)$. Then 
Lemma~\ref{lem:relint}.3 shows $F=F_K(x)=G$. This also shows that the map in 
question is injective. Its surjectivity is clear.
\end{proof}
\begin{Rem}[Partitions of extreme sets]~
\begin{enumerate}
\item
Every extreme set $E$ of $K$ is the union of the family 
$\{\rai(F_K(x)):x\in E\}$ by Coro.~\ref{cor:extreme}.2 and 
Thm.~\ref{thm:WeisShirokov}. This family is a partition of $E$ 
by Coro.~\ref{cor:partition}.
\item
There exist subfamilies of the partition $\mathfrak{U}_1$ in 
Coro.~\ref{cor:partition}, whose union is not an extreme set of $K$. An 
example is the subfamily having as its only element the open 
unit interval $(0,1)$ when $K=[0,1]$.
\item
If the present concept of an extreme set is replaced with that underlying 
Dubins' work \cite{Dubins1962}, then the union of any subfamily of 
$\mathfrak{U}_1$ will be an extreme set, cf.~Thm.~\ref{thm:d-extreme}.
\end{enumerate}
\end{Rem}
Coro.~\ref{cor:char-ri-F} and Coro.~\ref{cor:char-ri-Fx} are similar to 
\cite[Coro.~2.7]{WeisShirokov2021}. Proofs are provided for easy reference. 
\par
\begin{Cor}\label{cor:char-ri-F}
Let $F$ be a face of $K$ and let $x\in K$. The following statements are equivalent.
\begin{enumerate}
\item
We have $x\in\rai(F)$.
\item 
We have $F=F_K(x)$.
\item
We have $\rai(F)=\rai(F_K(x))$.
\end{enumerate}
\end{Cor}
\begin{proof}
1) $\Rightarrow$ 2) is Lemma~\ref{lem:relint}.3. 
2) $\Rightarrow$ 3) is clear.
3) $\Rightarrow$ 1) is implied by Thm.~\ref{thm:WeisShirokov}.
\end{proof}
\begin{Cor}\label{cor:char-ri-Fx}
Let $x,y\in K$. The following statements are equivalent.
\begin{enumerate}
\item
We have $x\in\rai(F(y))$.
\item 
We have $F_K(y)=F_K(x)$.
\item 
We have $\rai(F_K(y))=\rai(F_K(x))$.
\item
We have $x\in F_K(y)$ and $y\in F_K(x)$.
\end{enumerate}
\end{Cor}
\begin{proof}
1) $\Leftrightarrow$ 2) $\Leftrightarrow$ 3) is the special case of
Coro.~\ref{cor:char-ri-F} when $F$ is replaced with $F_K(y)$.
2) $\Rightarrow$ 4) is clear, and 4) $\Rightarrow$ 2) follows from
Lemma~\ref{lem:relint}.2.
\end{proof}
Prop.~\ref{pro:intersection} complements \cite[Thm.~4.3]{Dubins1962}
but is not equivalent to it, as different concepts of a ``face'' 
are in use (see Sec.~\ref{sec:Dubins-Face} below).
\par
\begin{Pro}\label{pro:intersection}
Let $K,L\subset V$ be convex sets and let $x\in K\cap L$. Then
\begin{enumerate}
\item
\parbox{5cm}{\raggedleft $F_{K\cap L}(x)$}
$=$
\parbox{5cm}{$F_K(x)\cap F_L(x)$\,\text{,}}
\item
\parbox{5cm}{\raggedleft $\rai\big(F_{K\cap L}(x)\big)$}
$=$
\parbox{5cm}{$\rai\big(F_K(x)\big)\cap\rai\big(F_L(x)\big)$\,\text{,}}
\item
\parbox{5cm}{\raggedleft $\aff\big(F_{K\cap L}(x)\big)$}
$=$
\parbox{5cm}{$\aff\big(F_K(x)\big)\cap\aff\big(F_L(x)\big)$\,\text{.}}
\end{enumerate}
\end{Pro}
\begin{proof}
See \cite[Prop.~2.13]{WeisShirokov2021}.
\end{proof}
\begin{Cor}\label{cor:face-in-intersection}
Let $K,L\subset V$ be convex sets and let $F$ be a nonempty face of $K\cap L$
with $\rai(F)\neq\emptyset$. Then $F$ is the intersection of a face of $K$ 
and a face of $L$. A sufficient condition for $\rai(F)\neq\emptyset$ is that 
$F$ have finite dimension.
\end{Cor}
\begin{proof}
Let $x\in\rai(F)$. Lemma~\ref{lem:relint}.3 shows $F=F_{K\cap L}(x)$ and 
Prop.~\ref{pro:intersection}.1 proves the first claim. If $\dim(F)<\infty$,
then $\mathfrak{T}_\omega(\aff(F))$ is the Euclidean topology on $\aff(F)$
\cite[II.26]{Bourbaki1987}. Hence, the relative interior of $F$ is nonempty 
by Thm.~6.2 in \cite{Rockafellar1970}. Coro.~\ref{cor:ri=rai} shows that the
relative algebraic interior is nonempty, too.
\end{proof}
Is the assumption of $\rai(F)\neq\emptyset$ in 
Coro.~\ref{cor:face-in-intersection} necessary?
\par
\begin{Ope}
Are there convex sets $K$ and $L$ and a face $F$ of $K\cap L$ that can not be 
written as the intersection of a face of $K$ and a face of $L$?
\end{Ope}
The analogue of Prop.~\ref{pro:intersection} for infinitely many convex sets
is wrong.
\par
\begin{Exa}
We consider the open segment $(-\epsilon,1+\epsilon)$ for every $\epsilon>0$. 
The intersection $\bigcap_{\epsilon>0}(-\epsilon,1+\epsilon)$ is the closed 
unit interval $[0,1]$. The only face of $(-\epsilon,1+\epsilon)$ containing 
the extreme point $0$ of $[0,1]$ is $(-\epsilon,1+\epsilon)$ itself. 
Therefore, $\{0\}$ is not an intersection of faces of 
$\{(-\epsilon,1+\epsilon)\}_{\epsilon>0}$, as such an intersection includes
$[0,1]$. More examples are obtained by replacing some or all of the open 
segments $(-\epsilon,1+\epsilon)$ with closed segments $[-\epsilon,1+\epsilon]$.
\end{Exa}
%
%%%%%%%%%%%%%%%%%%%%%%%%%%%%%%%%%%%%%%%%%%%%%%%%%%%%%%%%%%%%%%%%%%%%%%%%%%%%
%%%%%%%%%%%%%%%%%%%%%%%%%%%%%%%%%%%%%%%%%%%%%%%%%%%%%%%%%%%%%%%%%%%%%%%%%%%%
%%%%%%%%%%%%%%%%%%%%%%%%%%%%%%%%%%%%%%%%%%%%%%%%%%%%%%%%%%%%%%%%%%%%%%%%%%%%
%%%%%%%%%%%%%%%%%%%%%%%%%%%%%%%%%%%%%%%%%%%%%%%%%%%%%%%%%%%%%%%%%%%%%%%%%%%%
%%%%%%%%%%%%%%%%%%%%%%%%%%%%%%%%%%%%%%%%%%%%%%%%%%%%%%%%%%%%%%%%%%%%%%%%%%%%
%
\section{Novel results on faces and relative algebraic interiors}
\label{sec:novel-results-I}
This section is inspired by properties of relative interiors of convex 
sets in finite dimensions \cite{Rockafellar1970}.
Various insights into relative algebraic interiors are facilitated by
Prop.~\ref{pro:char2-open-segments} and Coro.~\ref{cor:ri-gen}. The unions 
in these statements extend over all open segments in $K$ containing $x$, 
and the singleton $\{x\}=(x,x)$. Prop.~\ref{pro:char2-open-segments} 
matches part of \cite[Thm.~6.1]{Rockafellar1970}.
\par
\begin{Pro}\label{pro:char2-open-segments}
For all $x\in K$ we have $\rai(F_K(x))=\bigcup_{y,z\in K,x\in(y,z)}(y,z)$.
\end{Pro}
\begin{proof}
The point $x$ lies in the left-hand side of the equation by 
Thm.~\ref{thm:WeisShirokov} and in the right-hand side by definition.
Let $a\in K$ and $a\neq x$. Then
\[\begin{array}{rcl}
a\in\rai(F_K(x))
 & \stackrel{\rm{Coro.~\ref{cor:char-ri-Fx}.4}}{\iff} 
 & \mbox{$a\in F_K(x)$ and $x\in F_K(a)$}\\
 & \stackrel{\rm{Coro.~\ref{cor:facex}}}{\iff}  
 & \exists y,z\in K : \mbox{$x\in(a,y)$ and $a\in(x,z)$}\\
 & \stackrel{a\neq x}{\iff}
 & \exists y,z\in K : x,a\in(y,z)
\end{array}\]
proves the claim.
\end{proof}
\begin{Cor}\label{cor:ri-gen}
Let $x\in K$. Then $\rai(K)=\bigcup_{y,z\in K,x\in(y,z)}(y,z)$ holds
if and only if $x\in\rai(K)$.
\end{Cor}
\begin{proof}
Coro.~\ref{cor:char-ri-F} shows that $x\in\rai(K)$ is equivalent to 
$\rai(K)=\rai(F_K(x))$. Now, the claim follows from 
Prop.~\ref{pro:char2-open-segments}.
\end{proof}
Coro.~\ref{cor:ri} matches part of \cite[Thm.~6.2]{Rockafellar1970}.
\par
\begin{Cor}\label{cor:ri}
If $\rai(K)\neq\emptyset$, then $\aff(\rai(K))=\aff(K)$ holds.
\end{Cor}
\begin{proof}
Let $x\in\rai(K)$ and $y\in\aff(K)$. By the definition of the relative 
algebraic interior, there is $\epsilon>0$ such that 
$x\pm\epsilon(x-y)\in K$. Coro.~\ref{cor:ri-gen} shows 
$y_\pm=x\pm\frac{\epsilon}{2}(x-y)\in\rai(K)$, which implies 
\[\textstyle
y=\left(\frac{1}{2}-\frac{1}{\epsilon}\right)y_+ 
+ \left(\frac{1}{2}+\frac{1}{\epsilon}\right)y_-
\in\aff(\rai(K))\,\text{.}
\]
The opposite inclusion is obvious.
\end{proof}
Coro.~\ref{cor:inter-rel-open} matches part of 
\cite[Thm.~6.5]{Rockafellar1970}. 
\par
\begin{Cor}\label{cor:inter-rel-open}
Let $K,L\subset V$ be convex sets and $\rai(K)\cap\rai(L)\neq\emptyset$.
Then
\[
\rai(K\cap L)=\rai(K)\cap\rai(L)
\quad\text{and}\quad
\aff(K\cap L)=\aff(K)\cap\aff(L)\,\text{.}
\]
The intersection of two relative algebraically open convex sets is 
relative algebraically open. 
\end{Cor}
\begin{proof}
Let $x\in\rai(K)\cap\rai(L)$. Then $K=F_K(x)$ and $L=F_L(x)$ by 
Lemma~\ref{lem:relint}.3 and the first claim follows from 
Prop.~\ref{pro:intersection}. If $K$ and $L$ are relative algebraically 
open, then 
\[
K\cap L 
=\rai(K)\cap\rai(L)
=\rai(K\cap L)
\]
shows that the intersection $K\cap L$ is relative algebraically open.
\end{proof}
Thm.~\ref{thm:ri-affine-trafo} matches part of
\cite[Thm.~6.6]{Rockafellar1970}.
\par
\begin{Thm}\label{thm:ri-affine-trafo}
Let $\rai(K)\neq\emptyset$ and let $\alpha:V\to W$ be an affine map 
to a real vector space $W$\!. Then $\rai(\alpha(K))=\alpha(\rai(K))$.
\end{Thm}
\begin{proof}
Let $x\in\rai(K)$. First, we show $\alpha(x)\in\rai(\alpha(K))$.
Let $y\neq\alpha(x)$ be a point in $\alpha(K)$ and choose any 
$y'\in\alpha|_K^{-1}(y)$. Since $x$ is an
internal point of $K$, there exists $\epsilon>0$ such that $x+\epsilon(x-y')$ 
lies in $K$. Applying $\alpha$ shows that $\alpha(x)+\epsilon(\alpha(x)-y)$ 
lies in $\alpha(K)$, hence $\alpha(x)$ is an internal point of $\alpha(K)$.
The implication 1)~$\Rightarrow$~2) of Prop.~\ref{pro:BorweinGoebel} 
implies $\alpha(x)\in\rai(\alpha(K))$. Second, Coro.~\ref{cor:ri-gen} proves
\begin{align*}
\alpha(\rai(K)) 
&=\textstyle
\alpha\left(\bigcup_{y,z\in K,x\in(y,z)}(y,z)\right)\\
&=\textstyle
\bigcup_{y,z\in K,x\in(y,z)}(\alpha(y),\alpha(z))\\
&=\textstyle
\bigcup_{y',z'\in\alpha(K),\alpha(x)\in(y',z')}(y',z')
\,\text{.}
\end{align*}
As $\alpha(x)\in\rai(\alpha(K))$, the last expression of this equation
equals $\rai(\alpha(K))$, again by Coro.~\ref{cor:ri-gen}.
\end{proof}
Thm.~\ref{thm:ri2=ri} matches Problem 2 in \cite[III.1.6]{Barvinok2002}.
\par
\begin{Thm}\label{thm:ri2=ri}
The set $\rai(K)$ is a relative algebraically open convex set.
\end{Thm}
\begin{proof}
The convexity of $\rai(K)$ is provided by Lemma~\ref{lem:ri-convex}. That 
$\rai(K)$ is relative algebraically open can be proved by showing for all
$x\in\rai(K)$ that $x\in\rai(\rai(K))$, or equivalently that $x$ is an 
internal point of $\rai(K)$, as per the implication 1)~$\Rightarrow$~2) of 
Prop.~\ref{pro:BorweinGoebel}. Let $y\neq x$ lie in $\rai(K)$.
Coro.~\ref{cor:ri-gen} provides $a,b\in K$ such that $x$ and $y$ are in the 
open segment $(a,b)$. Relabel $a$ and $b$, if necessary, such that 
$x=(1-\lambda)a+\lambda b$ and $y=(1-\mu)a+\mu b$ for scalars 
$0<\lambda<\mu<1$. Then $\frac{1}{2}(a+x)$ lies in $\rai(K)$, again 
by Coro.~\ref{cor:ri-gen}, and 
\[\textstyle
x+\frac{\lambda}{2(\mu-\lambda)}(x-y)=\frac{1}{2}(a+x)
\]
shows that $x$ is an internal point of $\rai(K)$.
\end{proof}
Coro.~\ref{cor:convex-subset} matches part of 
\cite[Thm.~2.1]{Dubins1962}, see also Coro.~\ref{cor:Dubins21} below.
\par
\begin{Cor}\label{cor:convex-subset}
For every $x\in K$, the set $\rai(F_K(x))$ is the greatest relative 
al\-ge\-bra\-ically open convex subset (with respect to inclusion) 
of $K$ that contains~$x$.
\end{Cor}
\begin{proof}
Thm.~\ref{thm:ri2=ri} shows that $\rai(F_K(x))$ is a relative algebraically 
open convex set, which contains $x$ by Thm.~\ref{thm:WeisShirokov}.
Let $C\subset K$ be relative algebraically open and convex, and let $x\in C$. 
For each $y\in C$, Coro.~\ref{cor:ri-gen} provides $a,b\in C$ such 
that $x,y\in(a,b)$. Hence, Prop.~\ref{pro:char2-open-segments} shows 
$y\in\rai(F_K(x))$. 
\end{proof}
Thm.~\ref{thm:maximal-elements} generalizes 
\cite[Thm.~18.2]{Rockafellar1970}.
\par
\begin{Thm}\label{thm:maximal-elements}
Let $K\neq\emptyset$. The family of maximal relative algebraically open 
convex subsets (with respect to inclusion) of $K$ is 
\[
\mathfrak{U}
=\left\{\rai(F_K(x)) \colon x \in K\right\}
=\left\{\rai(F)\colon\text{$F$ is a face of $K$}\right\}\setminus\{\emptyset\}
\,\text{.}
\]
The family $\mathfrak{U}$ is a partition of $K$. Each nonempty relative 
algebraically open convex subset of $K$ is included in a unique element 
of $\mathfrak{U}$.
\end{Thm}
\begin{proof}
Let $C$ be a nonempty relative algebraically open convex subset of $K$
and let $x$ be a point in $C$. Then $C\subset\rai(F_K(x))$ holds by 
Coro.~\ref{cor:convex-subset}. The set $C$ can not be included in any other 
element of $\mathfrak{U}$ because $\mathfrak{U}$ is a partition by 
Coro.~\ref{cor:partition}, which also proves the equality between the two 
descriptions of $\mathfrak{U}$.
\par
Let $C$ be a relative algebraically open convex subset of $K$. We show that the 
elements of $\mathfrak{U}$ are maximal. Assume that $C$ includes $\rai(F_K(x))$ 
for some $x\in K$. Then $C$ contains $x$ by Thm.~\ref{thm:WeisShirokov} and 
$C\subset\rai(F_K(x))$ follows from Coro.~\ref{cor:convex-subset}. We show that 
$K$ has no other maximal relative algebraically open convex subsets. Assume that 
$C$ is maximal. As $\mathfrak{U}\neq\emptyset$, the set $C$ contains some point 
$x\in K$. This implies $C=\rai(F_K(x))$ by Coro.~\ref{cor:convex-subset}.
\end{proof}
%
%%%%%%%%%%%%%%%%%%%%%%%%%%%%%%%%%%%%%%%%%%%%%%%%%%%%%%%%%%%%%%%%%%%%%%%%%%%%
%%%%%%%%%%%%%%%%%%%%%%%%%%%%%%%%%%%%%%%%%%%%%%%%%%%%%%%%%%%%%%%%%%%%%%%%%%%%
%%%%%%%%%%%%%%%%%%%%%%%%%%%%%%%%%%%%%%%%%%%%%%%%%%%%%%%%%%%%%%%%%%%%%%%%%%%%
%%%%%%%%%%%%%%%%%%%%%%%%%%%%%%%%%%%%%%%%%%%%%%%%%%%%%%%%%%%%%%%%%%%%%%%%%%%%
%%%%%%%%%%%%%%%%%%%%%%%%%%%%%%%%%%%%%%%%%%%%%%%%%%%%%%%%%%%%%%%%%%%%%%%%%%%%
%
\section{Novel results on generators of faces}
\label{sec:novel-results-II}
This section is motivated by theorems in \cite{Dubins1962}, but differs 
from them due to the nonequivalent concepts of a ``face'' 
(see Sec.~\ref{sec:Dubins-Face}).
\begin{Lem}\label{lem:riconv}
Let $K_i$ be a convex subset of $V$\!, let $x_i\in\rai(K_i)$, and let 
$\lambda_i>0$, $i=1,\ldots,n$, such that $\lambda_1+\ldots+\lambda_n=1$.
Then the point $\lambda_1x_1+\ldots+\lambda_nx_n$ lies in the relative 
algebraic interior of the convex hull of $K_1\cup\cdots\cup K_n$.
\end{Lem}
\begin{proof}
Let $n=2$, let $\lambda:=\lambda_2\in(0,1)$, $x:=(1-\lambda)x_1+\lambda x_2$,
and let $C$ denote the convex hull of $K_1\cup K_2$. Below, we construct 
for every $y\in C$ a number $\eta>0$ such that $x+\eta(x-y)$ lies in 
$C$. This shows that $x$ is an internal point of $C$, and the implication 
1)~$\Rightarrow$~2) of Prop.~\ref{pro:BorweinGoebel} proves $x\in\rai(C)$.
\par
As $y\in C$, there is $\mu\in[0,1]$ and there are $y_1\in K_1$ and $y_2\in K_2$
such that $y=(1-\mu)y_1+\mu y_2$. Since $x_i$ is an internal point of $K_i$, 
there exists $\epsilon>0$ such that $z_i:=x_i+\epsilon(x_i-y_i)$ lies in $K_i$, 
$i=1,2$. We distinguish two cases. If $\mu\leq\lambda$, then 
$\nu:=1-\mu +\epsilon(\lambda-\mu)$ and $\eta:=\epsilon(1-\lambda)/\nu$ are
positive and 
\[\textstyle
x+\eta(x-y)
=\frac{1}{\nu}\left[
\epsilon(\lambda-\mu)y_2
+(1-\lambda)(1-\mu )z_1
+\lambda(1-\mu )z_2
\right]\in K\,\text{.}
\]
If $\mu\geq\lambda$, then 
$\nu:=\mu+\epsilon(\mu-\lambda)$ and $\eta:=\epsilon\lambda/\nu$ are
positive and 
\[\textstyle
x+\eta(x-y)
=\frac{1}{\nu}\left[
\epsilon(\mu-\lambda)y_1
+(1-\lambda)\mu z_1
+\lambda\mu z_2
\right]\in K\,\text{.}
\]
Induction extends the claim from $n=2$ to all $n\in\N$.
\end{proof}
Thm.~\ref{thm:conv} matches \cite[(3.1)]{Dubins1962}.
\par
\begin{Thm}\label{thm:conv}
Let $C\subset K$ be convex. 
Then $\bigcup_{x\in C}F_K(x)$ is a face of $K$.
\end{Thm}
\begin{proof}
The set $E:=\bigcup_{x\in C}F_K(x)$ is a union of extreme sets, and hence an
extreme set itself. We show that $E$ is convex. Let $a_i\in E$ and let 
$c_i\in C$ such that $a_i\in F_K(c_i)$, $i=1,2$. By Coro.~\ref{cor:facex},
there is a point $b_i\in K$ such that $c_i$ lies in the relative algebraic 
interior of $[a_i,b_i]$, $i=1,2$. Lemma~\ref{lem:riconv} shows that the point
$c:=(c_1+c_2)/2$ lies in the relative algebraic interior of the convex 
hull $D$ of $\{a_1,a_2,b_1,b_2\}$. Since $C$ is convex, it contains $c$.
Hence $c$ lies in $E$, and Lemma~\ref{lem:relint}.1 proves $D\subset E$. 
It follows that $[a_1,a_2]\subset D\subset E$.
\end{proof}
We define the \emph{face generated by a subset $S\subset K$} as the 
smallest face of $K$ containing $S$. We denote this face by $F_K(S)$.
Coro.~\ref{cor:face-gen-S} matches \cite[(3.3)]{Dubins1962}.
\par
\begin{Cor}\label{cor:face-gen-S}
Let $S\subset K$. Then $F_K(S)=\bigcup_{x\in C}F_K(x)$, where $C$ is the 
convex hull of $S$.
\end{Cor}
\begin{proof}
The union $F:=\bigcup_{x\in C}F_K(x)$ is a face of $K$ by Thm.~\ref{thm:conv}. 
Let $G$ be any face containing $S$. As $G$ is convex it includes $C$ and it
also includes the face $F_K(x)$ for every $x\in C$ by Lemma~\ref{lem:relint}.2. 
This proves $F\subset G$.
\end{proof}
\begin{Cor}
Let $x_i\in K$, $i=1,\ldots,n$. Let $S=\bigcup_{i=1}^nF_K(x_i)$. Let 
$\lambda_i>0$, $i=1,\ldots,n$, such that $\lambda_1+\ldots+\lambda_n=1$, 
and put $x=\lambda_1x_1+\ldots+\lambda_nx_n$. Then
\[
F_K(S)
=F_K(\{x_1,\ldots,x_n\})
=F_K(x)\,\text{.}
\]
\end{Cor}
\begin{proof}
As the inclusions $F_K(S)\supset F_K(\{x_1,\ldots,x_n\})\supset F_K(x)$ 
follow from Coro.~\ref{cor:face-gen-S}, it suffices to prove 
$F_K(S)\subset F_K(x)$. Lemma~\ref{lem:riconv} shows that $x$ lies in 
the relative algebraic interior of the convex hull $C$ of 
$\{x_1,\ldots,x_n\}$. Hence $C$ is included in $F_K(x)$ by 
Lemma~\ref{lem:relint}.1. Lemma~\ref{lem:relint}.2 then shows that 
$F_K(x_i)$ is included in $F_K(x)$ for all $i=1,\ldots,n$, which proves 
$F_K(S)\subset F_K(x)$.
\end{proof}
Coro.~\ref{cor:faces-rel-open} matches \cite[(4.7)]{Dubins1962}.
\par
\begin{Cor}\label{cor:faces-rel-open}
Let $K,L\subset V$ be convex sets and let $K$ be relative algebraically open. Then every extreme 
set resp.\ face of $K\cap L$ is the intersection of $K$ and an extreme set resp.\ 
face of $L$. 
\end{Cor}
\begin{proof}
Let $E$ be an extreme set of $K\cap L$. Coro.~\ref{cor:extreme}.3 
and Prop.~\ref{pro:intersection}.1 show
\[\textstyle
E
=\bigcup_{x\in E}F_{K\cap L}(x)
=\textstyle\bigcup_{x\in E}\left(F_K(x)\cap F_L(x)\right)
\,\text{.}
\]
As $K$ is relative algebraically open, Lemma~\ref{lem:relint}.3 implies
$F_K(x)=K$ for all $x\in K$, hence
\[\textstyle
E=\textstyle\bigcup_{x\in E}\left(K\cap F_L(x)\right)
=\textstyle K\cap\bigcup_{x\in E}F_L(x)
\,\text{.}
\]
The set $\bigcup_{x\in E}F_L(x)$ is a union of extreme sets of $L$ and 
hence an extreme set of $L$ itself. If $E$ is a face of $K\cap L$, then 
$E$ is convex and Thm.~\ref{thm:conv} completes the proof.
\end{proof}
%
%%%%%%%%%%%%%%%%%%%%%%%%%%%%%%%%%%%%%%%%%%%%%%%%%%%%%%%%%%%%%%%%%%%%%%%%%%%%
%%%%%%%%%%%%%%%%%%%%%%%%%%%%%%%%%%%%%%%%%%%%%%%%%%%%%%%%%%%%%%%%%%%%%%%%%%%%
%%%%%%%%%%%%%%%%%%%%%%%%%%%%%%%%%%%%%%%%%%%%%%%%%%%%%%%%%%%%%%%%%%%%%%%%%%%%
%%%%%%%%%%%%%%%%%%%%%%%%%%%%%%%%%%%%%%%%%%%%%%%%%%%%%%%%%%%%%%%%%%%%%%%%%%%%
%%%%%%%%%%%%%%%%%%%%%%%%%%%%%%%%%%%%%%%%%%%%%%%%%%%%%%%%%%%%%%%%%%%%%%%%%%%%
%
\section{Dubins' terminology}
\label{sec:Dubins-Face}
A \emph{d-extreme set} of $K$ is a subset $E$ of $K$ including the open segment 
$(x,y)$ for all points $x\neq y$ in $K$ for which $(x,y)$ intersects $E$. A 
point $x\in K$ is a \emph{d-extreme point} of $K$ if $\{x\}$ is a d-extreme set 
of $K$. A \emph{d-face} of $K$ is a convex d-extreme set of $K$. Clearly, any 
union or intersection of d-extreme sets of $K$ is a d-extreme set of $K$. 
Hence, the intersection of all d-faces containing a point $x\in K$ is a d-face 
of $K$, which we call the \emph{d-face of $K$ generated by $x$}. This is the 
smallest d-face of $K$ containing $x$.
\par
\begin{Thm}\label{thm:d-extreme}
A subset of $K$ is a d-extreme set of $K$ if and only if it is equal to the 
union $\bigcup_{x\in S}\rai(F_K(x))$ for some subset $S\subset K$. The d-face 
generated by $x\in K$ is $\rai(F_K(x))$. A point in $K$ is a d-extreme point 
of $K$ if and only if it is an extreme point of $K$.
\end{Thm}
\begin{proof}
Let $E$ be a d-extreme set of $K$ and let $x\in E$. As $x\in\rai(F_K(x))\subset E$ 
holds by Prop.~\ref{pro:char2-open-segments}, we have 
$E=\bigcup_{x\in E}\rai(F_K(x))$. Conversely, let $S\subset K$ be any subset and 
assume that a point $y$ lies in the open segment $(a,b)$ with endpoints $a\neq b$ 
in $K$ and in the set $\rai(F_K(x))$ for some $x\in S$. The open segment $(a,b)$ 
is included in $\rai(F_K(y))$ by Coro.~\ref{cor:convex-subset} and hence in 
$\rai(F_K(x))$, as $\rai(F_K(y))=\rai(F_K(x))$ holds by Coro.~\ref{cor:char-ri-Fx}.
\par
The set $\rai(F_K(x))$ contains $x$ by Thm.~\ref{thm:WeisShirokov} and is convex 
by Lemma~\ref{lem:ri-convex}. Hence, it is the smallest d-face containing $x$ by 
the first part of this theorem.
\par
That ``d-extreme point'' and ``extreme point'' are equivalent terms is 
implied by the fact that a singleton cannot contain a segment no matter 
whether it is an open segment or a closed segment.
\end{proof}
Coro.~\ref{cor:Dubins21} provides an alternative proof of 
\cite[Thm.~2.1]{Dubins1962}.
\par
\begin{Cor}\label{cor:Dubins21}
For all $x\in K$, the d-face of $K$ generated by $x$ is equal to $\rai(F_K(x))$ 
and equal to the greatest relative algebraically open convex subset of $K$ 
containing $x$.
\end{Cor}
\begin{proof}
Thm.~\ref{thm:d-extreme} shows that $\rai(F_K(x))$ is the d-face of $K$ generated 
by $x$. That $\rai(F_K(x))$ is the greatest relative algebraically open convex subset of $K$ 
containing $x$ is proved in Coro.~\ref{cor:convex-subset}.
\end{proof}
By definition \cite{Dubins1962}, an \emph{elementary d-face} of $K$ is a
nonempty relative algebraically open d-face of $K$. 
\par
\begin{Cor}\label{cor:elementary-d-faces}
Let $K\neq\emptyset$. A subset of $K$ is an elementary d-face of $K$ 
if and only if it equals $\rai(F_K(x))$ for some $x\in K$.
\end{Cor}
\begin{proof}
Let $F$ be a d-face of $K$. By Thm.~\ref{thm:d-extreme}, $F$ is a union of a 
subfamily of $\mathfrak{U}:=\left\{\rai(F_K(x)) \colon x \in K\right\}$. 
If $F$ is an elementary d-face of $K$, then $F$ cannot be a union of more than 
one element of $\mathfrak{U}$, because the elements of $\mathfrak{U}$ are the 
maximal relative algebraically open convex subsets of $K$ by 
Thm.~\ref{thm:maximal-elements}.
\end{proof}
\begin{Cor}\label{cor:2d-faces}
Let $K\neq\emptyset$. The following statements are equivalent.
\begin{enumerate}
\item
$K$ is relative algebraically open.
\item
$K$ has exactly two d-faces (which are $\emptyset$ and $K$).
\item
$K$ has exactly two faces (which are $\emptyset$ and $K$).
\end{enumerate}
\end{Cor}
\begin{proof}
Assuming 1), $\rai(K)=K$ is nonempty. Hence, the partition of $K$
\[
\mathfrak{U}=\left\{\rai(F_K(x)) \colon x \in K\right\}
\]
contains $\rai(K)$ as one of its elements by Coro.~\ref{cor:partition}.
Hence, $\emptyset$ and $K$ are the only d-faces of $K$
by Thm.~\ref{thm:d-extreme}.
\par
The statement 2) implies that $K$ has at most two faces, as every face of 
$K$ is a d-face of $K$. Since $K\neq\emptyset$, the convex set $K$ has exactly 
two faces.
\par
Assuming 3), the convex set $K$ has only one nonempty face. Then $\{\rai(K)\}$ 
is a partition of $K$ by Coro.~\ref{cor:partition}, which implies $K=\rai(K)$.
\end{proof}
%
%
%%%%%%%%%%%%%%%%%%%%%%%%%%%%%%%%%%%%%%%%%%%%%%%%%%%%%%%%%%%%%%%%%%%%%%%%%%%%
%%%%%%%%%%%%%%%%%%%%%%%%%%%%%%%%%%%%%%%%%%%%%%%%%%%%%%%%%%%%%%%%%%%%%%%%%%%%
%%%%%%%%%%%%%%%%%%%%%%%%%%%%%%%%%%%%%%%%%%%%%%%%%%%%%%%%%%%%%%%%%%%%%%%%%%%%
%%%%%%%%%%%%%%%%%%%%%%%%%%%%%%%%%%%%%%%%%%%%%%%%%%%%%%%%%%%%%%%%%%%%%%%%%%%%
%%%%%%%%%%%%%%%%%%%%%%%%%%%%%%%%%%%%%%%%%%%%%%%%%%%%%%%%%%%%%%%%%%%%%%%%%%%%
%
\section{Examples 1: Spaces of probability measures}
\label{sec:example-prob-measures}
Let $\cP=\cP(\Omega,\cA)$ denote the convex set of probability measures on a 
measurable space $(\Omega,\cA)$. A probability measure $\lambda\in\cP$ is 
\emph{absolutely continuous} with respect to $\mu\in\cP$, symbolically 
$\lambda\ll\mu$, if every $\mu$-null set is a $\lambda$-null set. The 
measures are \emph{equivalent}, $\lambda\equiv\mu$, if $\lambda\ll\mu$ and 
$\mu\ll\lambda$. If $\lambda\ll\mu$, then we denote by 
$\frac{\dr\lambda}{\dr\mu}:\Omega\to[0,\infty)$ the 
\emph{Radon-Nikodym derivative} of $\lambda$ with respect to $\mu$, which is a 
measurable function satisfying 
$\lambda(A)=\int_A \frac{\dr\lambda}{\dr\mu}\,\dr\mu$ for all $A\in\cA$, see 
for example Halmos \cite[Sec. 31]{Halmos1974}. 
\par
\begin{Lem}
For every $\mu\in\cP$, the set $\{\lambda\in\cP:\lambda\ll\mu\}$ is a face
of $\cP$.
\end{Lem}
\begin{proof}
Let $D(\mu):=\{\lambda\in\cP:\lambda\ll\mu\}$.
The set $D(\mu)$ is an extreme set of $\cP$. Let $\lambda_1,\lambda_2\in\cP$, 
$s\in(0,1)$, and $\lambda:=(1-s)\lambda_1+s\lambda_2$ be contained in $D(\mu)$.
If $A$ is a $\mu$-null set, then $A$ is a $\lambda$-null set and hence 
a $\lambda_i$-null set for $i=1,2$.
\par
The set $D(\mu)$ is convex, as $(1-s)\lambda_1+s\lambda_2$ has the Radon-Nikodym 
derivative $(1-s)\frac{\dr\lambda_1}{\dr\mu}+s\frac{\dr\lambda_2}{\dr\mu}$ with 
respect to $\mu$ for all $\lambda_1,\lambda_2\in D(\mu)$ and $s\in[0,1]$. 
\end{proof}
If $\mu\in\cP$, then we say a proposition $\pi(\omega)$, $\omega\in\Omega$, 
is true \emph{$\mu$-almost surely}, which we abbreviate as \emph{$\mu$-a.s.}, 
if $\mu(\{\omega\in\Omega \colon \text{$\pi(\omega)$ is false} \})=0$.
\par
\begin{Thm}\label{thm:faces-propm}
Let $\lambda,\mu\in\cP$. The following assertions are equivalent.
\begin{enumerate}
\item
The measure $\lambda$ lies in the face $F_\cP(\mu)$ of $\cP$ generated by $\mu$.
\item
There is $c\in[1,\infty)$ such that $\lambda(A)\leq c\,\mu(A)$ holds for all 
$A\in\cA$.
\item
We have $\lambda\ll\mu$ and there is $c\in[1,\infty)$ such that
$\frac{\dr\lambda}{\dr\mu}\leq c$ holds $\mu$-a.s..
\end{enumerate}
\end{Thm}
\begin{proof}
Prop.~\ref{pro:Alfsen} shows that a probability measure $\lambda\in\cP$ lies in 
$F_\cP(\mu)$ if and only if there is $\epsilon>0$ such that 
$\mu+\epsilon(\mu-\lambda)\in\cP$. The latter condition is equivalent to the 
nonnegativity of the set function $\mu+\epsilon(\mu-\lambda)$, and hence to 
part 2) of the theorem. It remains to prove the equivalence  
2)~$\Leftrightarrow$~3).
\par
If $\lambda\ll\mu$ and if there is $c\geq 1$ such that 
$\frac{\dr\lambda}{\dr\mu}(\omega)\leq c$ holds $\mu$-a.s.,
then part~2) follows (with the same constant $c$),
\[\textstyle
\lambda(A)=\int_A \frac{\dr\lambda}{\dr\mu}\,\dr\mu
\leq c\,\mu(A)
\quad 
\text{for all $A\in\cA$.}
\]
Conversely, if $\lambda\ll\mu$ is false, then there is $A\in\cA$ such that
$\lambda(A)>\mu(A)=0$, making part~2) impossible. If $\lambda\ll\mu$ is true
but $\frac{\dr\lambda}{\dr\mu}$ is not bounded $\mu$-a.s., then 
for every $c>0$ there is $A\in\cA$ such that $\mu(A)>0$ and 
$\frac{\dr\lambda}{\dr\mu}(\omega)>c$ holds for all $\omega\in A$. 
Then
\[\textstyle
\lambda(A)=\int_A \frac{\dr\lambda}{\dr\mu}\,\dr\mu
>c\,\mu(A)
\]
proves that part 2) fails. 
\end{proof}
Dubins in \cite{Dubins1962} asserts that a probability measure $\lambda\in\cP$
is contained in the smallest d-face of $\cP$ generated by $\mu\in\cP$ if and 
only if $\lambda\ll\mu$ and there exists $c>0$, such that for all $A\in\cA$ we 
have $\mu(A)\leq c\,\lambda(A)\leq c^2\mu(A)$. Coro.~\ref{cor:Dubins21} 
translates this assertion into Coro.~\ref{cor:ri-faces-propm}.
\par
\begin{Cor}[Dubins]\label{cor:ri-faces-propm}
Let $\lambda,\mu\in\cP$. The following statements are equivalent.
\begin{enumerate}
\item
The measure $\lambda$ lies in $\rai(F_\cP(\mu))$.
\item
There is $c\geq 1$ such that $\mu(A)/c\leq\lambda(A)\leq c\,\mu(A)$ 
for all $A\in\cA$.
\item
We have $\lambda\equiv\mu$ and there are $c_1,c_2\in[1,\infty)$ 
such that $\frac{\dr\lambda}{\dr\mu}\leq c_1$ holds $\mu$-a.s.\
and $\frac{\dr\mu}{\dr\lambda}\leq c_2$ holds $\lambda$-a.s..
\item
We have $\lambda\ll\mu$ and there is $c\in[1,\infty)$ such that
$\frac{1}{c}\leq \frac{\dr\lambda}{\dr\mu}\leq c$ 
holds $\mu$-a.s..
\end{enumerate}
\end{Cor}
\begin{proof}
Coro.~\ref{cor:char-ri-Fx} shows that $\lambda\in\rai(F_\cP(\mu))$ is equivalent 
to $\lambda\in F_\cP(\mu)$ and $\mu\in F_\cP(\lambda)$, so the equivalences  
1)~$\Leftrightarrow$~2)~$\Leftrightarrow$~3) follow from those of 
Thm.~\ref{thm:faces-propm}. 
\par
Note that $\lambda$-a.s.\ is the same as $\mu$-a.s.\ if  $\lambda\equiv\mu$. 
Hence, 3) implies that
$\frac{\dr\mu}{\dr\lambda}
\cdot\frac{\dr\lambda}{\dr\mu}=1$
and hence
$\frac{\dr\lambda}{\dr\mu}
=(\frac{\dr\mu}{\dr\lambda})^{-1}\geq1/c_2$
holds $\mu$-a.s., see for example \cite[Thm.~A, p.~133]{Halmos1974}.
Conversely, if $\lambda\ll\mu$ and if there is $c\in[1,\infty)$ such 
that $\frac{1}{c}\leq \frac{\dr\lambda}{\dr\mu}$ holds $\mu$-a.s., 
then
\[\textstyle
\lambda(A)=\int_A \frac{\dr\lambda}{\dr\mu}\,\dr\mu
\geq\frac{1}{c}\,\mu(A)
\quad 
\text{for all $A\in\cA$}
\]
implies that $\mu\equiv\lambda$ and that
$\frac{\dr\mu}{\dr\lambda}=
(\frac{\dr\lambda}{\dr\mu})^{-1}
\leq c$ holds $\lambda$-a.s..
\end{proof}
%
%
%%%%%%%%%%%%%%%%%%%%%%%%%%%%%%%%%%%%%%%%%%%%%%%%%%%%%%%%%%%%%%%%%%%%%%%%%%%%
%%%%%%%%%%%%%%%%%%%%%%%%%%%%%%%%%%%%%%%%%%%%%%%%%%%%%%%%%%%%%%%%%%%%%%%%%%%%
%%%%%%%%%%%%%%%%%%%%%%%%%%%%%%%%%%%%%%%%%%%%%%%%%%%%%%%%%%%%%%%%%%%%%%%%%%%%
%%%%%%%%%%%%%%%%%%%%%%%%%%%%%%%%%%%%%%%%%%%%%%%%%%%%%%%%%%%%%%%%%%%%%%%%%%%%
%%%%%%%%%%%%%%%%%%%%%%%%%%%%%%%%%%%%%%%%%%%%%%%%%%%%%%%%%%%%%%%%%%%%%%%%%%%%
%
\section{Examples 2: Convex cores}
\label{sec:example-convex-cores}
In a second example, we consider the Borel $\sigma$-algebra $\cB(d)$ of 
$\R^d$. The \emph{convex core} $\cc(\mu)$ of $\mu\in\cP=\cP(\R^d,\cB(d))$ is 
the intersection of all convex sets $C\in\cB(d)$ of full measure 
$\mu(C)=\mu(\R^d)$. The convex core was introduced in \cite{CsiszarMatus2001} 
to extend exponential families in a natural way, such that information 
projections become properly defined. The \emph{mean} of $\mu$ is the integral 
$m(\mu)=\int_{\R^d}x\dr\mu(x)\in\R^d$, provided that each coordinate function 
is $\mu$-integrable; otherwise, $\mu$ does not have a mean.
\par
\begin{Thm}[Csiszár and Matú\v{s}]\label{thm:csiszar-matus}
Let $\mu\in\cP(\R^d,\cB(d))$ have a mean. Then the convex core of $\mu$ equals
$\cc(\mu)=m\left(\left\{\lambda\in\cP\colon\lambda\ll\mu\right\}\right)$.
Moreover, to each $a\in\cc(\mu)$ there exists $\lambda\in\cP$ with 
$\lambda\ll\mu$ and mean $m(\lambda)=a$ such that $\frac{\dr\lambda}{\dr\mu}$ 
is bounded $\mu$-a.s..
\end{Thm}
Thm.~\ref{thm:csiszar-matus} is proved in Thm.~3 of \cite{CsiszarMatus2001}.
We derive from it a description of the relative algebraic interior of the 
convex core.
\par
\begin{Cor}\label{cor:csiszar-matus}
Let $\mu\in\cP(\R^d,\cB(d))$ have a mean. Then $\cc(\mu)=m(F_\cP(\mu))$. 
The relative algebraic interior of $\cc(\mu)$ is 
$\rai(\cc(\mu))=m(\rai(F_\cP(\mu)))$, which equals
\[\textstyle
\rai(\cc(\mu))
=m\left(\left\{\lambda\in\cP\mid
\lambda\equiv\mu, \exists c\in(1,\infty)\colon 
\frac{1}{c}\leq \frac{\dr\lambda}{\dr\mu}\leq c \mbox{~$\mu$-a.s.}
\right\}\right)
\,\text{.}
\]
\end{Cor}
\begin{proof}
Thm.~\ref{thm:faces-propm} and Thm.~\ref{thm:csiszar-matus} show that
$\cc(\mu)=m(F_\cP(\mu))$. As $\mu$ lies in the relative algebraic interior 
of $F_\cP(\mu)$ by Thm.~\ref{thm:WeisShirokov}, we obtain
\[
\rai(\cc(\mu))=m(\rai(F_\cP(\mu)))
\]
from Thm.~\ref{thm:ri-affine-trafo}. Coro.~\ref{cor:ri-faces-propm}
completes the proof.
\end{proof}
The characterization of $\rai(\cc(\mu))$ in Coro.~\ref{cor:csiszar-matus} is 
somewhat stronger than that in Lemma~5 of \cite{CsiszarMatus2001}, which 
ignores the lower bound $0<\frac{1}{c}\leq \frac{\dr\lambda}{\dr\mu}$ $\mu$-a.s..
Lemma~5 of \cite{CsiszarMatus2001} also shows 
$\rai(\cc(\mu))=m\left(\left\{\lambda\in\cP\colon\lambda\equiv\mu\right\}\right)$,
which cannot be deduced from Thm.~\ref{thm:csiszar-matus} with the methods 
developed in this paper, without the assistance of other methods.
\par
%
%%%%%%%%%%%%%%%%%%%%%%%%%%%%%%%%%%%%%%%%%%%%%%%%%%%%%%%%%%%%%%%%%%%%%%%%%%%%
%%%%%%%%%%%%%%%%%%%%%%%%%%%%%%%%%%%%%%%%%%%%%%%%%%%%%%%%%%%%%%%%%%%%%%%%%%%%
%%%%%%%%%%%%%%%%%%%%%%%%%%%%%%%%%%%%%%%%%%%%%%%%%%%%%%%%%%%%%%%%%%%%%%%%%%%%
%%%%%%%%%%%%%%%%%%%%%%%%%%%%%%%%%%%%%%%%%%%%%%%%%%%%%%%%%%%%%%%%%%%%%%%%%%%%
%%%%%%%%%%%%%%%%%%%%%%%%%%%%%%%%%%%%%%%%%%%%%%%%%%%%%%%%%%%%%%%%%%%%%%%%%%%%
%
\section{Examples 3: Discrete probability measures}
\label{sec:example-countable}
In a third example, we consider the discrete $\sigma$-algebra $2^\N$ of all 
subsets of $\N$. Consider the Banach space 
$\ell^1=\{x:\N\to\C\mid \|x\|_1<\infty\}$ of ab\-so\-lute\-ly summable 
sequences endowed with the $\ell^1$-norm $\|x\|_1=\sum_{n=1}^\infty|x(n)|$. 
We study $\cP=\cP(\N,2^\N)$ in terms of the set of probability mass functions 
\[\textstyle
\Delta_\N=\left\{p:\N\to\R\mid \forall n\in\N\colon p(n)\geq 0 
\mbox{~and~}\|p\|_1=1\right\}\,\text{.}
\]
The map $\cP\to\Delta_\N$ that maps a probability measure $\mu\in\cP$ to its 
Radon-Nikodym derivative $\frac{\dr\mu}{\dr\nu}$ with respect to the counting 
measure $\nu$, is an affine isomorphism. We endow $\cP$ with the 
\emph{distance in variation} 
\[\textstyle
\|\mu-\lambda\|
:=2\sup_{A\subset\N}|\mu(A)-\lambda(A)|
\,\text{,}
\quad 
\lambda,\mu\in\cP
\,\text{.}
\]
Then $\cP\to\Delta_\N$ is an isometry, as
$\|\mu-\lambda\|=\|\frac{\dr\mu}{\dr\nu}-\frac{\dr\lambda}{\dr\nu}\|_1$ holds 
\cite[Sec.~3.9]{Shiryaev2016}. 
The \emph{support} of $p\in\Delta_\N$ is $\spt(p):=\{n\in\N\colon p(n)>0\}$.
Lemma~\ref{lem:faces-pmf} is proved in \cite[Lemma 2.12]{WeisShirokov2021},
and follows also from Thm.~\ref{thm:faces-propm} and 
Coro.~\ref{cor:ri-faces-propm}. 
\par
\begin{Lem}\label{lem:faces-pmf}
For every $p\in\Delta_\N$ we have
\begin{align*}\textstyle
F_{\Delta_\N}(p) &= \textstyle
\left\{q\in\Delta_{\spt(p)}\colon \sup_{n\in\spt(p)}q(n)/p(n)<\infty\right\}
\,\text{,}\\
\rai(F_{\Delta_\N}(p)) &= \textstyle
\left\{q\in F_{\Delta_\N}(p)\colon \inf_{n\in\spt(p)}q(n)/p(n)>0\right\}
\,\text{.}
\end{align*}
\end{Lem}
\begin{Exa}[Faces with empty relative algebraic interiors]\label{ex:faceI}
For all subsets $I\subset\N$ we define
\[
\Delta_I :=\{p\in\Delta_\N\colon \spt(p)\subset I\}
\quad\text{and}\quad
\Delta_{I,\text{fin}} :=\{p\in\Delta_I\colon \nu(\spt(p))<\infty\}
\,\text{.}
\]
The sets $\Delta_I$ and $\Delta_{I,\text{fin}}$ are faces of $\Delta_\N$,
see \cite[Example 2.10]{WeisShirokov2021}. If $I$ is an infinite set, then 
$\rai(\Delta_I)=\rai(\Delta_{I,\text{fin}})=\emptyset$ holds by 
Lemma~\ref{lem:relint}.3, as each of the faces $\Delta_I$ and 
$\Delta_{I,\text{fin}}$ is strictly larger than the face generated by any 
of its points $p$. This is clear if $J:=\spt(p)$ is finite. Otherwise, 
if $J$ is infinite, let $p_H(n):=p(n)/(\sqrt{r_n}+\sqrt{r_{n+1}})$ for 
$n\in\N$, where $r_n:=\sum_{m\geq n}p(m)$. Then $p_H\in\Delta_J\subset\Delta_I$ 
but $p_H\not\in F_{\Delta_\N}(p)$ holds by Lemma~\ref{lem:faces-pmf}.
\end{Exa}
In \cite[Sec.~2]{WeisShirokov2021}, we raise the question as to whether 
$\Delta_\N$ has other faces with empty relative algebraic interiors, aside 
from those described in Example~\ref{ex:faceI}. 
Example~\ref{exa:infinite-chain} shows the answer is yes. 
\par
\begin{Exa}[Another face with empty relative algebraic interior]%
\label{exa:infinite-chain}
For every $s>1$, let $p_s:\N\to\R$, $n\mapsto\zeta(s)^{-1}\cdot n^{-s}$, where 
$\zeta(s)=\sum_{n\in\N}n^{-s}$ is the Euler-Riemann zeta function. For all
$s,t>1$, Lemma~\ref{lem:faces-pmf} shows that the probability mass function 
$p_s$ is included in $F_{\Delta_\N}(p_t)$ if and only if $s\geq t$. Hence, 
Lemma~\ref{lem:relint}.2 proves
\[
F_{\Delta_\N}(p_s) \subset F_{\Delta_\N}(p_t)
\iff
s\geq t
\,\text{,}
\qquad s,t>1
\,\text{.}
\]
By Lemma~\ref{lem:infinite-chain} below, for every $t\geq 1$,
the union
\[\textstyle
F_t:=\bigcup_{s>t}F_{\Delta_\N}(p_s)
\]
is a face of $\Delta_\N$ and $\rai(F_t)=\emptyset$. Let $t>1$. Then $F_t$ 
is included in $F_{\Delta_\N}(p_t)$, which is properly included in 
$\Delta_\N$ by Ex.~\ref{ex:faceI}. We also have $F_t\neq\Delta_{I,\text{fin}}$ 
and $F_t\neq\Delta_I$ for all $I\subset\N$, because $F_t$ contains the point 
$p_{t+1}$ of support $\N$.
\end{Exa}
\begin{Lem}\label{lem:infinite-chain}
Let $\{x_\alpha\}_{\alpha\in A}$ be a set of points in a convex set $K$ 
indexed by a totally ordered set $A$ that has no greatest element, 
such that $\alpha\leq\beta$ if and only if 
$F_K(x_\alpha)\subset F_K(x_\beta)$ holds for all $\alpha,\beta\in A$. 
Then $F=\bigcup_{\alpha\in A}F_K(x_\alpha)$ is a face of $K$ and 
$\rai(F)=\emptyset$.
\end{Lem}
\begin{proof}
The set $F$ is convex. If $a,b\in F$, then $a\in F_K(x_\alpha)$ and 
$b\in F_K(x_\beta)$ for some $\alpha,\beta\in A$. Both points $x_\alpha$
and $x_\beta$ lie in $F_K(x_{\max(\alpha,\beta)})$, hence the closed 
segment with endpoints $a,b$ lies in $F_K(x_{\max(\alpha,\beta)})\subset F$. 
The set $F$ is an extreme set. If the open segment with endpoints $a\neq b$ 
in $K$ intersects $F$, then it intersects $F_K(x_\alpha)$ for some 
$\alpha\in A$. It follows that $a,b\in F_K(x_\alpha)\subset F$.
\par
Assume there is $a\in\rai(F)$. Since $a\in F$, there is $\alpha\in A$ 
with $a\in F_K(x_\alpha)$. Lemma~\ref{lem:relint}.1 implies 
$F\subset F_K(x_\alpha)$, which shows that $\alpha$ is the greatest 
element of $A$. This is excluded from the assumptions.
\end{proof}
Whereas $\Delta_\N$ is closed in the $\ell^1$-norm \cite{Werner2018}, the 
face $F_{\Delta_\N}(p)$ is not closed for any $p\in\Delta_\N$ of infinite 
support. Indeed, Lemma~\ref{lem:norm-closed-faces} shows that the closure 
of $F_{\Delta_\N}(p)$ is $\Delta_{\spt(p)}$ whereas $F_{\Delta_\N}(p)$ is 
strictly included in $\Delta_{\spt(p)}$ by Ex.~\ref{ex:faceI}. Let the 
function $e_n\in\ell^1$ be defined by $e_n(m)=1$ if $n=m$ and $e_n(m)=0$ 
otherwise, $m,n\in\N$.
\par
\begin{Lem}\label{lem:norm-closed-faces}
The closure of any face $F$ of $\Delta_\N$ in the $\ell^1$-norm is 
$\Delta_{I(F)}$, where $I(F):=\bigcup_{p\in F}\spt(p)$. For every 
$p\in\Delta_\N$ we have $I(F_{\Delta_\N}(p))=\spt(p)$.
\end{Lem}
\begin{proof}
First, the face 
$\Delta_{I,\text{fin}}$ of functions with finite support is dense in 
$\Delta_I$ for all $I\subset\N$. To see this, let $I=\N$ (without loss 
of generality) and let $p\in\Delta_\N$. Then 
$(p_k)_{k\in\N}\subset\Delta_{\N,\text{fin}}$, defined by
\[
p_k(n)=
\left\{\begin{array}{ll}
p(n) & \text{if $n<k$,}\\
\sum_{m\geq k}p(m) & \text{if $n=k$,}\\
0 & \text{else,}
\end{array}\right.
\quad k,n\in\N
\,\text{,}
\]
converges to $p$, as $\|p-p_k\|_1=2\sum_{m>k}p(m)$ for all $k\in\N$.
\par
Second, if $F$ is a face of $\Delta_\N$ then 
$\Delta_{I(F),\text{fin}}\subset F\subset\Delta_{I(F)}$ holds. The right 
inclusion is obvious. To see the left inclusion, let $n\in I(F)$, and let 
$p\in F$ such that $n\in\spt(p)$. By Lemma~\ref{lem:faces-pmf}, we have 
$e_n\in F_{\Delta_\N}(p)$. Then $e_n\in F$ follows from 
Lemma~\ref{lem:relint}.2. This implies $\Delta_{I(F),\text{fin}}\subset F$ 
as $F$ is convex.
\par
The preceding two arguments prove the first assertion. The second assertion 
is a special case of the first one and follows from Lemma~\ref{lem:faces-pmf}.
\end{proof}
Lemma~\ref{lem:norm-closed-faces} shows that the norm closed
faces of $\Delta_\N$ are in a one-to-one correspondence with the subsets of 
$\N$. This assertion is a special case of a more general property of von Neumann 
algebras \cite{AlfsenShultz2001}. The space $\ell^1$ is the predual of the 
von Neumann algebra 
\[\textstyle
\ell^\infty=\{x:\N\to\C\mid \sup_{n\in\N}|x(n)|<\infty\}
\,\text{.}
\]
The set $\Delta_\N\subset\ell^1$ is the \emph{normal state space} of $\ell^\infty$, 
and the subsets $I$ of $\N$ are in a one-to-one correspondence with the projections 
in $\ell^\infty$, that is to say, functions $\N\to\{0,1\}$. In a general von 
Neumann algebra, there is an order preserving isomorphism between the norm closed 
faces of the normal state space and the projections in the algebra 
\cite[Thm.~3.35]{AlfsenShultz2001}. 
\par
%
%%%%%%%%%%%%%%%%%%%%%%%%%%%%%%%%%%%%%%%%%%%%%%%%%%%%%%%%%%%%%%%%%%%%%%%%%%%%
%%%%%%%%%%%%%%%%%%%%%%%%%%%%%%%%%%%%%%%%%%%%%%%%%%%%%%%%%%%%%%%%%%%%%%%%%%%%
%%%%%%%%%%%%%%%%%%%%%%%%%%%%%%%%%%%%%%%%%%%%%%%%%%%%%%%%%%%%%%%%%%%%%%%%%%%%
%%%%%%%%%%%%%%%%%%%%%%%%%%%%%%%%%%%%%%%%%%%%%%%%%%%%%%%%%%%%%%%%%%%%%%%%%%%%
%%%%%%%%%%%%%%%%%%%%%%%%%%%%%%%%%%%%%%%%%%%%%%%%%%%%%%%%%%%%%%%%%%%%%%%%%%%%
%
\subsection*{Acknowledgements} 
I thank Didier Henrion, Martin Kru\v{z}ík, Milan Korda, Milan Studený,
and Tobias Fritz for inspiring discussions. This work was co-funded 
by the European Union under the project ROBOPROX 
(reg. no. CZ.02.01.01/00/22\_008/0004590).
\par
%
%%%%%%%%%%%%%%%%%%%%%%%%%%%%%%%%%%%%%%%%%%%%%%%%%%%%%%%%%%%%%%%%%%%%%%%%%%%%
%%%%%%%%%%%%%%%%%%%%%%%%%%%%%%%%%%%%%%%%%%%%%%%%%%%%%%%%%%%%%%%%%%%%%%%%%%%%
%%%%%%%%%%%%%%%%%%%%%%%%%%%%%%%%%%%%%%%%%%%%%%%%%%%%%%%%%%%%%%%%%%%%%%%%%%%%
%%%%%%%%%%%%%%%%%%%%%%%%%%%%%%%%%%%%%%%%%%%%%%%%%%%%%%%%%%%%%%%%%%%%%%%%%%%%
%%%%%%%%%%%%%%%%%%%%%%%%%%%%%%%%%%%%%%%%%%%%%%%%%%%%%%%%%%%%%%%%%%%%%%%%%%%%
%
\bibliographystyle{plain}

%
%
%%%%%%%%%%%%%%%%%%%%%%%%%%%%%%%%%%%%%%%%%%%%%%%%%%%%%%%%%%%%%%%%%%%%%%%%%%%%
%%%%%%%%%%%%%%%%%%%%%%%%%%%%%%%%%%%%%%%%%%%%%%%%%%%%%%%%%%%%%%%%%%%%%%%%%%%%
%%%%%%%%%%%%%%%%%%%%%%%%%%%%%%%%%%%%%%%%%%%%%%%%%%%%%%%%%%%%%%%%%%%%%%%%%%%%
%%%%%%%%%%%%%%%%%%%%%%%%%%%%%%%%%%%%%%%%%%%%%%%%%%%%%%%%%%%%%%%%%%%%%%%%%%%%
%%%%%%%%%%%%%%%%%%%%%%%%%%%%%%%%%%%%%%%%%%%%%%%%%%%%%%%%%%%%%%%%%%%%%%%%%%%%
%
\vspace{\baselineskip}
\parbox{10cm}{%
Stephan Weis\\
Czech Technical University in Prague\\ 
Faculty of Electrical Engineering\\
Karlovo nám\v{e}stí 13\\
12000, Prague 2\\
Czech Republic\\
e-mail \texttt{maths@weis-stephan.de}}
\end{document}